\documentclass[a4paper,oneside,11pt]{article}
\usepackage{amsmath,amsthm,amsfonts,amssymb}
\usepackage[utf8]{inputenc}
\usepackage[T1]{fontenc}
\usepackage{lmodern}
\usepackage{graphicx}

\long\def\beginpgfgraphicnamed#1#2\endpgfgraphicnamed{\includegraphics{#1}} 

\usepackage[text={16cm,24cm},centering]{geometry}
\newcounter{stepcnt}

\newcommand*{\abs}[1]{\left|#1\right|} 
\newcommand*{\field}[1]{\mathbb{#1}}

\let\xR\R

\newcommand*{\xN}{\mathbb{N}}

\newcommand*{\anynorm}[1]{\left\|#1\right\|}
\newcommand*{\norm}[2]{\anynorm{#2}_{#1}}

\newcommand*{\compl}[1]{#1^{c}}

\newcommand*{\ind}[1]{\mathbf{1}_{#1}}

\DeclareMathOperator{\supp}{Supp}

\DeclareMathOperator{\cpc}{Cap}

\newcommand{\var}{\mathbf{Var}}

\newcommand*{\ie}{\emph{i.e.}}
\newcommand*{\eg}{\emph{e.g.}}
\newtheorem{thrm}{Theorem}
\newtheorem{prpstn}[thrm]{Proposition}
\newtheorem{dfntn}[thrm]{Definition}
\newtheorem{rmrk}[thrm]{Remark}

\newtheorem{lmm}[thrm]{Lemma}
\newtheorem{claim}[thrm]{Claim}

\newcommand{\PhiStar}{\Phi^{\star}}
\newcommand{\bstar}{{b^\star}}
\newcommand*{\dmu}{\,d\mu}

\title{Super Poincaré inequalities, Orlicz norms and essential spectrum}
\author{P.-A. Zitt}

\begin{document}
\maketitle

\begin{abstract}
  We prove some results about the super Poincaré inequality (SPI) and its relation to the spectrum of an operator:  we show that it can be alternatively written with Orlicz norms instead of $L^1$ norms, and we use this to give an alternative proof that  a bound on the bottom of the essential spectrum implies a SPI. Finally, we apply these ideas to give a spectral proof of the log Sobolev inequality for the Gaussian measure.

  Keywords: Super Poincaré inequality, Essential spectrum, measure--capacity inequality
  
  MSC 2010: 39B62, 47A10, 47D07
\end{abstract}

\section{Introduction}
In recent years, the study of various functional inequalities, and their relation to decay estimates for various processes, has attracted much attention. Some of these inequalities, namely the Poincaré inequality and the super Poincaré inequality, have a direct relation with the spectrum of the generator of the process; for others the relation with spectral properties is not so clear (see however \cite{ABD07}, where the spectrum is used to prove functional inequalities).  Our focus here will be the super Poincaré inequality. 

Before presenting our results, let us explain the setting. For the sake of clarity, we choose the simplest possible setting and consider a diffusion operator in $\xR^n$: 
\[ L: f \mapsto \Delta f - \nabla V \cdot \nabla f,\]
where $V$ is a smooth function such that $\exp(-V)$ is integrable. In this setting, $L$ is the generator of a diffusion with gradient drift, which has an invariant and reversible  measure $\mu(dx) = \exp(-V(x)) dx$: $L$ is self-adjoint in $L^2(d\mu)$. Finally  we define the Dirichlet form $D(f,f)$ by
\[ D(f,f) = \int f (-L)f \dmu = \int \abs{\nabla f}^2 \dmu.\]

In \cite{Wan00Essential}, Wang introduces a new functional inequality, which he calls the ``super Poincaré inequality'' (SPI, see definition \ref{def=SPI} below). 
  He proves that this inequality is qualitatively equivalent to the absence of essential spectrum for the generator $L$. In subsequent papers (in particular \cite{GW02}, with Gong, see also the review paper \cite{Wan00Semigroup}), he studies the links between this inequality, the essential spectrum, the compactness of various resolvents and bounds on the semi-group generated by $L$. 

 We will be particularly interested in the proof of a functional inequality, starting from information on the spectrum. Wang's first proof of this result (\cite{Wan00Essential}) relies on a decomposition of the space: he uses a spectral result (Donnelly--Li's decomposition principle) to show that the essential spectrum depends on the behaviour of $L$ at infinity; then he studies the functional inequality on the inside and on the outside of a large ball. 

 In \cite{GW02}  a more ``natural'' proof of  this implication is given, using the spectral decomposition of the generator (\ie{} $L$ is decomposed on the frequency space rather than the ordinary space). 

 \bigskip
 We present here two results around the SPI. In section \ref{sec=SPIimpliesPI} we recall definitions and known facts about this inequality. Then we introduce a  variant of it, where we replace an $L^1$ norm  by a more general Orlicz norm; we show in section \ref{sec=OSPI} that this variant is qualitatively equivalent to the original SPI. In section \ref{sec=WangRevisited}, we use this variant to give an alternative proof of the fact that bounds on the essential spectrum of the operator $L$ imply a super Poincaré inequality.  Finally, in section \ref{sec=example},  we show how to recover the log Sobolev inequality for the gaussian from the spectral decomposition of the generator of the Ornstein--Uhlenbeck process.

\section{The super Poincaré inequality}
\label{sec=SPIimpliesPI}
\subsection{Definitions}
Let us first recall the classical Poincaré inequality: 
\begin{dfntn}
  The measure $\mu$ satisfies the Poincaré inequality if there is a $C$ such that:
  \begin{equation}
    \label{eq=defPoincare}
  \forall f,\quad \var_\mu (f) \leq C \int \abs{\nabla f}^2 \dmu.
\end{equation}
\end{dfntn}
The definition of the super Poincaré inequality is somewhat obscured by a different choice of notation between \cite{Wan00Essential} and \cite{BCR07}: we follow the first article and use the following definition.
\begin{dfntn}
  [Wang, \cite{Wan00Essential}]
  \label{def=SPI}%
  A measure $\mu$  satisfies a Super Poincaré inequality if: 
  \begin{equation}
    \forall r>0, \exists \beta(r), \forall f, \quad
    \int f^2 \dmu \leq r \int \abs{\nabla f}^2 \dmu + \beta(r) \left( \int \abs{f} \dmu \right)^2.
    \label{eq=defSPI}
  \end{equation}
  We say that a partial SPI holds if \eqref{eq=defSPI} is satisfied only for $r>r_0$.
\end{dfntn}

  In \cite{BCR07}  a different (if similar) definition is used, which we temporarily denote by SPI(BCR). These authors consider the inequality:
  \begin{equation}
    \label{eq=SPIBCR}
     \forall s\geq 1, \exists \beta_{BCR}(s), \forall f \quad \int f^2 \dmu \leq \beta_{BCR}(s) \int \abs{\nabla f}^2 \dmu + s \left( \int \abs{f} \dmu \right)^2.
  \end{equation}
   This turns out to be equivalent to the formulation of Wang, in the following manner.
   \begin{prpstn}
     SPI (BCR) is equivalent to a partial SPI (in the sense of definition \ref{def=SPI}), with $r_0 = \lim_{s\to \infty} \beta_{BCR}(s)$. 
   \end{prpstn}
   Indeed, $\beta_{BCR}$ in \eqref{eq=SPIBCR} is not a priori supposed to go to zero, so one can only get \eqref{eq=defSPI} for $r$ big enough. Conversely, if \eqref{eq=defSPI} holds, one could a priori only show \eqref{eq=SPIBCR} for $s \geq \lim_{r\to \infty} \beta(r)$. However, a (possibly partial) SPI implies Poincaré (as we will reprove below), therefore one can always set $\beta(r) = 1$ in \eqref{eq=defSPI} for $r$ large enough; so \eqref{eq=SPIBCR} holds for every $s\geq 1$.

\subsection{Relation with the Poincaré inequality}
As a warm-up, we give a proof of the Poincaré inequality, given a SPI and an additional hypothesis on the function $\beta$.
\begin{rmrk}
Let us note here that this hypothesis may be replaced by much more natural ones: in fact, the SPI can be used together with a weaker inequality (such as local Nash inequalities, \cite{Wan00Essential}, th. 2.1, or weak Poincaré inequalities, \cite{RW01}) to get the Poincaré inequality. The use of the weak Poincaré inequality makes for a very nice functional proof. 
\end{rmrk}

\begin{thrm}
  Suppose $\mu$ satisfies a (possibly partial) SPI, and that the  following condition holds: $\exists r>0, \beta(r)<2$.

  Then $\mu$ satisfies the Poincaré inequality. 
\end{thrm}

The proof will use the following elementary lemma. 
\begin{lmm}
  \label{lmm=lemmeSympa}
  \begin{enumerate}
    \item For all $f$, $\left( \int \abs{f} \dmu \right)^2 = \norm{1}{f}^2 \leq \mu(\supp f) \norm{2}{f}^2$.
    \item If the SPI holds, then
      \begin{equation}
	\forall r>r_0, \forall f, \quad \left( 1 - \beta(r) \mu(\supp(f))\right) \int f^2 \dmu \leq r \int \abs{\nabla f }^2 \dmu.
    \end{equation}
  \end{enumerate}
\end{lmm}
\begin{proof}
  The lemma is almost trivial:
  the first part is Hölder's inequality; the second part follows directly from the first and the definition of the SPI.

To prove the theorem, let us first suppose that for some $r$, $\beta(r)<2$. 
Let $f$ be in $L^2$, and let $m_f$ be a median of $f$. 
Then
\begin{equation}
  \label{eq=coupeEnDeux}
  \var(f)\leq  \int (f- m_f)^2 \dmu \leq  \int f_+^2 \dmu + \int f_-^2 \dmu,
\end{equation}
where $f_\pm = (f-m_f)_\pm$.
We apply lemma \ref{lmm=lemmeSympa} to $f_+$ and $f_-$: since $\mu(\supp(f_\pm)) \leq 1/2$ by definition of the median, we get
\begin{equation}
  \label{eq=passeAGauche}
  (1 - \beta(r)/2) \int f_\pm^2 \dmu \leq r \int \abs{\nabla f_\pm}^2 \dmu 
\end{equation}
We choose $r$ such that  $\beta(r)<2$, and define $C = r / (1 - \beta(r)/2)$. Then we
 inject \eqref{eq=passeAGauche} (written for $f_+$ and $f_-$) into \eqref{eq=coupeEnDeux}, and we get:
\begin{align*}
  \var_\mu(f)
  &\leq C \left( \int \abs{\nabla f_+}^2 \dmu + \int \abs{\nabla f_-}^2 \dmu\right) \\
  &\leq C \int \abs{\nabla f}^2 \dmu,
\end{align*}
which is the Poincaré inequality \eqref{eq=defPoincare}
\end{proof}

\section{Super Poincaré with Orlicz norms}
\label{sec=OSPI}
\subsection{Measure--capacity inequalities}
Looking at the definition of the SPI (eq. \eqref{eq=defSPI}), it is natural to ask if we have any freedom in the choice of the norm in the right hand side. For weak Poincaré inequalities, for example, we have shown in \cite{Zit07annealing} (section 4) that one can work with various norms instead of a sup norm. Here the situation is similar: any norm ``between'' $L^1$ and $L^2$ should give (qualitatively) the same properties. 
In order to study these generalized inequalities, we would like to have a characterisation of them in terms of a measure--capacity criterion. This follows the well established approach of \cite{BR03} (for details, see for example \cite{BCR05b}, section 5.2), where inequalities between measure and capacity of sets are used as a general benchmark to compare various functional inequalities. Let us briefly recall the specific notion of capacity involved:
\begin{dfntn}
  For any set $A$ whose measure is smaller than $1/2$, the capacity of $A$ is defined by: 
  \begin{equation}
  \label{eq=defCapacity}
  \cpc_\mu(A) = \inf \left\{ \int \abs{\nabla f}^2 \dmu,
  \ind{A} \leq f \leq 1, 
  \mu(\supp(A)) \leq 1/2\right\}.
\end{equation}
  The functions appearing in this definition will be called ``admissible'' for $A$.
\end{dfntn}
We will use the following comparison between capacity and measures of sets:
\begin{dfntn}
  We say that the measure--capacity inequality (MC) holds for $(\kappa,C_\kappa)$ if, for all sets $A$ such that $\mu(A)\leq \kappa$, 
  \begin{equation}
    \label{eq=MC}
    \cpc_\mu(A) \geq C_\kappa \mu(A).
\end{equation}
\end{dfntn}
Inequalities like MC are known to be related to functional inequalities. For example, if $C_{MC}$ is the optimal constant in the inequality above for $\kappa = 1/2$, and $C_P$ is the optimal Poincaré constant, then
\begin{equation}
  \label{eq=CMPoincare}
  C_{MC}^{-1} \leq C \leq 4 C_{MC}^{-1}.
\end{equation}
(see \cite{BCR05b}, proposition 13).
\subsection{From Super Poincaré to measure--capacity}
The SPI has already been studied with measure--capacity tools (\cite{BCR07}). 
With our notations, we obtain:
\begin{thrm}
  \label{thm=CMSPI}
  If MC holds for $(\kappa,C_\kappa)$, and if $\kappa\mapsto \frac{\kappa}{C_\kappa}$ is non increasing,  then the SPI \eqref{eq=defSPI} holds with:
  \begin{equation}
    \label{eq=fromMCtoSPI}
    \beta(r) = \frac{1}{ \inf\{ \kappa, C_\kappa \geq 8/r\} },
\end{equation}
whenever this quantity is finite. 
  In particular, the SPI holds for $r > (8\cdot  \lim_{\kappa\to 0} C_\kappa^{-1})$, so we have a full SPI if $C_\kappa \to \infty$. 
\end{thrm}
\begin{rmrk}
  This can be compared with known results on the weak Poincaré inequality. In \cite{Zit07annealing}, we showed
  that the WPI is equivalent to the existence of a $C_\kappa$ such that \eqref{eq=MC} holds for sets \emph{larger} than $\kappa$. In other words, the capacity of large sets is controlled by their measure, and the constant becomes worse ($C_\kappa \to 0$) when $\kappa$ gets smaller. 

  Here the situation si reversed: there is a constant $C$ such that \eqref{eq=MC} holds for any $A$; however, if we restrict our attention to small sets, the constant gets better ($C_\kappa \to \infty$).

  Finally, note that if MC holds for any $(\kappa,C_\kappa)$, we can always find $C_\kappa'\leq C_\kappa$ such that $\kappa \mapsto \frac{\kappa}{C_\kappa'}$ is non increasing, and MC holds for $(\kappa, C_\kappa')$. 
\end{rmrk}
\begin{rmrk}
  The factor $8$ in equation \eqref{eq=fromMCtoSPI} is unfortunate. It could probably be changed into $4$ (by looking carefully at the proofs in \cite{BCR07}). The loss is probably inherent to our technique: even in the simpler case of the Poincaré inequality, one looses a constant factor when going from Poincaré to the measure--capacity inequality, and then back. Indeed, the constant $4$ in \eqref{eq=CMPoincare} is optimal (see the remarks on proposition $13$ in \cite{BCR05b}, which refers to \cite{Maz85}). 
\end{rmrk}
\begin{proof}
  [Proof of theorem \ref{thm=CMSPI}] This is a rephrasing of theorem 1 and corollary $6$ from \cite{BCR07}. Indeed, if our hypotheses are satisfied, and we define $\beta_{BCR}(s) = (C_{1/s})^{-1}$, then $\beta_{BCR}(s)$ is non-increasing, $s\mapsto s\beta_{BCR}(s)$ is non decreasing, and, for all $A$, 
  \[ \cpc_\mu(A) \geq \frac{\mu(A)}{\beta(1/\mu(A))}.\]
  Thanks to corollary $6$ of \cite{BCR07}, this implies a SPI with
  \[ \beta(r) = \inf \{ s, 8\beta_{BCR}(s) \leq r \}.\]
  Since $\beta_{BCR}(s) = (C_{1/s})^{-1}$, we obtain \eqref{eq=fromMCtoSPI}.
\end{proof}

We now give a converse statement to theorem \ref{thm=CMSPI}, and prove a measure capacity inequality, starting from an SPI. 

\begin{thrm}
  \label{thm=SPIimpliesMC}
  Suppose that the (partial) S.P.I. \eqref{eq=defSPI} holds (for $r>r_0$). Then, for any $\kappa$, there is a $C_\kappa$ such that \eqref{eq=MC} holds:
  \[ \forall A, \mu(A) \leq \kappa \implies  \cpc_\mu(A) \geq C_\kappa \mu(A). \]

  Moreover, one can choose $C_\kappa$ such that $C_\kappa \to  \frac{1}{r_0}$ when $\kappa \to 0$. In particular, $C_\kappa \to \infty$ if the full SPI holds.
\end{thrm}
\begin{proof}
  Let us first note that, since SPI implies a Poincaré inequality, we already know that the measure capacity inequality \eqref{eq=MC} holds, with $C_\kappa = C$ independent of $\kappa$; therefore our goal is to prove a better estimate on the capacity of small sets. 

  The idea is as follows: consider a ``small'' set $A$. To estimate its capacity, we study functions $f$ such that $\ind{A}\leq f \leq 1$. Let us consider a ``good'' function, \ie{} one who approaches the infimum in the definition of capacity (eq. \eqref{eq=defCapacity} ). 
  Two cases are possible: either $f$ stays close to $1$ on a large set ($B$), or it decreases sharply, just outside $A$ (see figure \ref{fig=fTilde}). In the first case, $f$ resembles an admissible function for the large set, therefore
  \[
  \cpc_\mu(A) \approx \int \abs{\nabla f}^2 d\mu \geq \cpc_\mu(B) \geq C_{\mu(B)} \cdot \mu(B) \gg C_{\mu(B)} \mu(A),
  \]
  and we are in good shape. In the second case, since $f$ decreases sharply, its ``support'' is morallly small, and we can use the SPI to get a good bound on the capacity, in the vein of lemma \ref{lmm=lemmeSympa}. 

  \begin{figure}
    \caption{The two cases for the definition of $\tilde{f}$\label{fig=fTilde}}
    \beginpgfgraphicnamed{figure1}
\begin{tikzpicture}
  \draw[very thick,|-|] (-2,0) --  node [below] {$A$} (0,0);
  \draw[very thick,|-|] (-5,-0.7) -- node [below] {$B$} (2,-0.7);
  \draw[thick] (-9,0) -- (-7,0)
  to [out=right,in=190]
  (-5,1.75)
  to [out=10,in=left]
  (-2,2) -- (0,2)
  to [out=right,in=160]
  (2,1.75)
  to [out=-20,in=left]
  (5,0) -- (6,0);
  \draw[densely dashed] (-9,0)--(-7,0)
  to [out=right,in=191.4]
  (-5,2)--(2,2)
  to [out=-22.8, in=left]
  (5,0);
  \draw [help lines] (-5,1.75)-- (5,1.75) node [below right,black] {$b$};
  \draw [help lines] (-7,0)   -- (5,0);
  \draw [help lines] (-2,2)   -- (5,2)    node [right,black] {$1$};

\end{tikzpicture}
\endpgfgraphicnamed

\beginpgfgraphicnamed{figure2}
\begin{tikzpicture}
  \draw[very thick,|-|] (-2,0)    -- node [below,black] {$A$} (0,0);
  \draw[very thick,|-|] (-5,-0.7) -- node [below,black] {$B$} (2,-0.7);
  \draw[thick] (-9,0) -- (-7,0)
  to [out=right,in=210]
  (-5,0.6)
  to [out=30,in=left]
  (-2,2) -- (0,2)
  to [out=right,in=150]
  (2,0.6)
  to [out=-30,in=left]
  (5,0) -- (6,0);
  \draw[densely dashed] (-5,0)
  to [out=42.9,in=left]
  (-2,2)--(0,2)
  to [out=right, in=137.14]
  (2,0);
  \draw [help lines] (-5,0.6)-- (5,0.6) node [right,black] {$b$};
  \draw [help lines] (-7,0)  -- (5,0);
  \draw [help lines] (-2,2)  -- (5,2) node [right,black] {$1$};

\end{tikzpicture}
\endpgfgraphicnamed

    \small
    The function $f$ is the solid line, $\tilde{f}$ is the dashed line. In the first case, $\tilde{f}$ is an admissible function for the large set $B$. In the second case, $f$ decays sharply outside $A$, so it is well approximated by a function $\tilde{f}$ with small support.
  \end{figure}
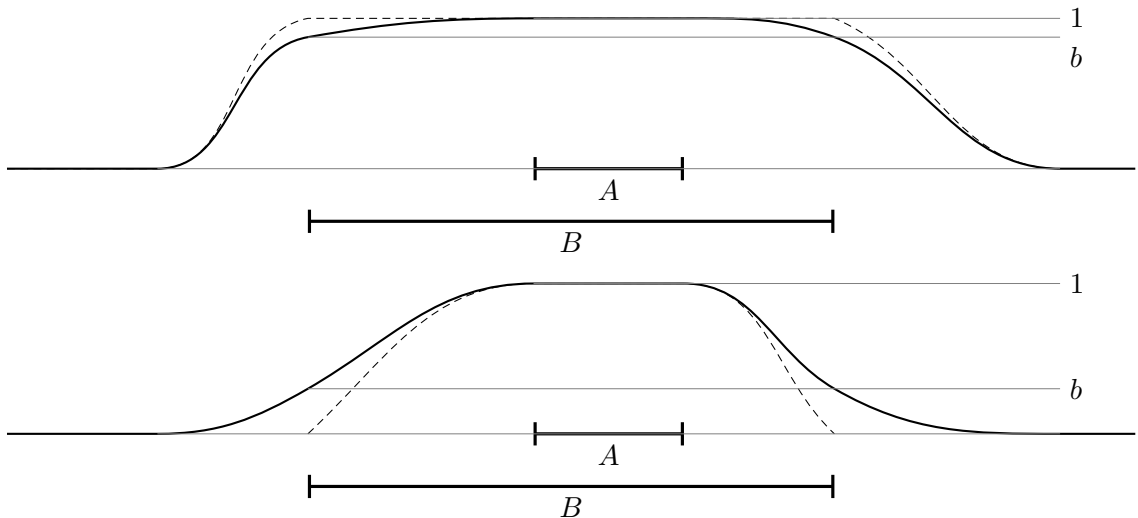
  Let us now be more formal. Let $\psi$ be any function such that $\psi(x) \to 0$, $\psi(x)/x \to \infty$ when $x$ goes to zero (\eg{} $\psi(x) = x\log(1/x$: this particular choice will be used later). Fix $\kappa$, and consider  a set $A$ such that $\mu(A) \leq \kappa$. Let $f$ be an admissible function for $A$. Finally, consider the level sets $B = \{ f\geq b\}$, and choose $b$ such that $\mu(B) = \psi(\kappa)$ (therefore $B$ is large, compared to $A$).

  Let us fix a threshold $\bstar \in (0,1)$.  If $b\geq \bstar$, $f$ stays (relatively) large on the ``large set'' $B$. Let $\tilde{f} = f/b$: $\tilde{f}$  is admissible for $B$, therefore
  \begin{align}
    \int \abs{\nabla{f}}^2 \dmu 
    &\geq b^2\int_{\Omega\setminus B}\abs{ \nabla{\tilde{f}}}^2
    &\text{(def. of $\tilde{f}$)} \notag \\
    &\geq b^2\cpc_\mu(B)
    &\text{($\tilde{f}$ is admissible for $B$)} \notag\\
    &\geq b^2 C \psi(\kappa)
    &\text{(M--C for $B$, with constant $C$)} \notag\\
    &\geq (\bstar)^2 \frac{C\psi(\kappa)}{\kappa} \kappa \notag \\
    \label{eq=minor_cap_1}
    &\geq (\bstar)^2 \frac{C\psi(\kappa)}{\kappa} \mu(A).
  \end{align}

  If $b\leq \bstar$, $f$ decreases sharply: we use the SPI. Define $\tilde{f} = (f-b)/(1-b) $ on $B$, $\tilde{f}=0 $ on $\compl{B}$ (see figure \ref{fig=fTilde}). For this choice, $\mu(\supp(\tilde{f})) \leq \mu(B) = \psi(\kappa)$, and $\mu(\tilde{f}^2) \geq \mu(A)$. Thanks to lemma \ref{lmm=lemmeSympa}, we get, for any $r>r_0$:
  \begin{equation}
    \int \abs{\nabla \tilde{f}}^2 \dmu \geq \frac{1}{r} \left( 1 - \beta(r) \psi(\kappa) \right) \cdot \mu(A)
    \label{eq=utilisationLemmeSympa}
  \end{equation}
  Going back to $f$, we obtain:
  \begin{align}
    \notag
    \int \abs{\nabla{f}}^2 \dmu 
    &\geq \int_B \abs{\nabla{f}}^2 \dmu  & \\
    \notag
    &\geq (1-b)^2\int_B \abs{\nabla \tilde{f}}^2 \dmu  & \text{(def. of $\tilde{f}$)}\\
    &\geq (1-b)^2 \frac{1 - \beta(r) \psi(\kappa)}{r} \mu(A).
  \end{align}
  Since $b$ is smaller than $\bstar$ and $r$ is arbitrary, 
  \begin{equation}
    \int \abs{\nabla f}^2 \dmu 
    \geq (1 - \bstar)^2 \left(\sup_{r > r_0} \frac{ 1 - \beta(r) \psi(\kappa)}{r}  
    \right)\mu(A).
    \label{eq=minor_cap_2}
  \end{equation}

  Now for any admissible $f$, one of \eqref{eq=minor_cap_1} or \eqref{eq=minor_cap_2} holds:
  \[ \int \abs{\nabla f}^2 \dmu \geq 
   \min\left(
    \frac{C\psi(\kappa)}{\kappa} \bstar^2 ,
    \sup_{r>r_0} \frac{ 1 - \beta(r) \psi(\kappa)}{r}(1 - \bstar)^2 
  \right) \cdot \mu(A).
  \]
  Taking the infimimum over all admissible $f$, we show that \eqref{eq=MC} holds, and we can choose
  \[ C_\kappa = 
   \min\left(
    \frac{C\psi(\kappa)}{\kappa} \bstar^2 ,
    \sup_{r>r_0} \frac{ 1 - \beta(r) \psi(\kappa)}{r}(1 - \bstar)^2 
  \right)
  \]
  The first term in the minimum goes to infinity when $\kappa \to 0$, by our choice of $\psi$. 
  Since $\psi(\kappa) \to 0$, it is easy to see that $\sup_{r>r_0} \frac{1 - \beta(r) \psi(\kappa)^2}{r}$ goes to $\frac{1}{r_0}$ when $\kappa \to 0$. Therefore, $\liminf_{\kappa \to 0} C_\kappa \geq (1 - \bstar)^2 \frac{1}{r_0}$. Since $\bstar$ is arbitrary, we may send it to zero and get the final result.
\end{proof}

\subsection{Super Poincaré with Orlicz norms}
We are now ready to prove our main result about super Poincaré inequalities with Orlicz norms.

Let us briefly recall basic facts about Orlicz spaces (see appendix A in \cite{RZ07} for more details).
An even, convex function satisfying $\Phi(0) = 0$, $\lim_{x\to\infty} \Phi(x) = \infty$  is called a Young function. To any such function we associate the Orlicz space
\[ L_\Phi = \left\{ f, \exists \lambda, \int \Phi(\abs{f}/\lambda) \dmu \leq 1\right\}.\]
If we define the ``Luxembourg norm'' $\norm{\Phi}{f}$ to be the smallest possible $\lambda$ in the previous equation, $(L_\Phi,\norm{\Phi}{\cdot})$ becomes  a Banach space. This norm is equivalent to  the Orlicz norm, defined by duality:
\[ \mathcal{N}_\Phi(f) = \sup \left\{ \int fg \dmu, g\in M_{\PhiStar} \right\} \]
where $M_{\PhiStar} = \left\{ g, \int \PhiStar(\abs{g})\dmu \leq 1\right\}$, and $\PhiStar$ is the Legendre transform of $\Phi$. More precisely, 
\begin{equation}
  \label{eq=eqNormes}
  \forall f \in L_\Phi, \quad \norm{\Phi}{f} \leq \mathcal{N}_\Phi(f) \leq 2 \norm{\Phi}{f}.
\end{equation}
We now state our result.
\begin{thrm}
  \label{thm=OSPI}
  Suppose that $(\Phi,\PhiStar)$ is a dual pair of finite Young functions, satisfying
  $\lim_{x\to\infty}\frac{\PhiStar(x)}{x^2} = \infty$. Suppose additionally that $x\mapsto\PhiStar( \sqrt{x})$ is also a Young function. 

  Then the (partial) Orlicz SPI:
  \begin{equation}
    \forall r>r_0, \quad 
    \int f^2 \dmu \leq  r \int \abs{\nabla f}^2 \dmu + \beta(r) \norm{\Phi}{f}^2
    \label{eq=defOSPI}
  \end{equation}
  implies a (partial) SPI, valid for $r>8r_0$, and with another function $\tilde{\beta}$, which depends explicitely on $\beta$ and $\Phi$. 
\end{thrm}
The growth conditions on $\Phi$ imply the following inclusions between spaces:
\[ L^\infty \subsetneq L_{\PhiStar} \subsetneq L^2 \subsetneq L_{\Phi} \subsetneq L^1.\]
In a sense, we have generalized the SPI for any ``reasonable'' norm between $L^1$ and $L^2$. 
The proof of theorem \ref{thm=OSPI} relies on the following generalization of lemma \ref{lmm=lemmeSympa}:
\begin{lmm}
  \label{lmm=lemmeSympaOrlicz}
  Let $\Phi$ and $\PhiStar$ be as in theorem \ref{thm=OSPI}.  
  Then
  \begin{equation}
    \label{eq=lemmeSympaOrliczUn}
  \norm{\Phi}{f} \leq \norm{2}{f} \theta(\mu(\supp f)),
\end{equation}
  where $\theta(x) = \frac{2}{\PhiStar\left(\sqrt{x^{-1}}\right)^{1/2}} \to 0$ when $x\to 0$.  If additionally, the Orlicz--SPI \eqref{eq=defOSPI} holds, then 
  \begin{equation}
    \label{eq=lemmeSympaOrliczDeux}
  \left( 1 - \beta(r) \theta^2\left(\mu(\supp f)\right)\right) \int f^2 \dmu 
  \leq r \int \abs{\nabla f}^2 \dmu.
\end{equation}
\end{lmm}
\begin{proof}
  [Proof of the lemma]
  Let $A$ be the support of $f$. 
  By \eqref{eq=eqNormes} and the definition of the Orlicz norm, 
  \[ \norm{\Phi}{f} \leq \mathcal{N}_\Phi(f) 
  = \sup \left\{ \int fg \dmu, g \in M_{\Phi^*} 
  \right\}.
  \]
  Let us apply Hölder's inequality: 
  \begin{align*}
    \norm{\Phi}{f} 
    &\leq \norm{2}{f} \sup \left\{ \sqrt{\int g^2 \ind{A} \dmu}, g \in M_{\Phi^*} \right\} \\
    &\leq \norm{2}{f} \sup \left\{ \int h   \ind{A} \dmu, h \in M_{\Phi^*\circ\sqrt{\cdot}}\right\}^{1/2} \\
    &\leq \norm{2}{f} \left(\mathcal{N}_{\PhiStar\circ\sqrt{\cdot}}(\ind{A})\right)^{1/2} \\
    &\leq 2\norm{2}{f}\left( \norm{(\Phi^*\circ \sqrt{\cdot})^*}{\ind{A}}\right)^{1/2}.
  \end{align*}
  The Luxembourg norm of an indicatrix function is known, and depends only on the measure of the set: 
  \[\norm{(\Phi^*\circ \sqrt{\cdot})^*}{\ind{A}} = \frac{1}{\PhiStar(\sqrt{\mu(A)^{-1}})}.\]
  This proves the first inequality \eqref{eq=lemmeSympaOrliczUn}. Since $\Phi$ is finite, $\lim_{x\to\infty} \PhiStar(x)  = \infty $, so $\lim_{x\to 0}\theta(x) = 0$. 
  The second inequality is then a direct consequence of \eqref{eq=lemmeSympaOrliczUn} and  \eqref{eq=defOSPI}.
\end{proof}

\begin{proof}[Proof of theorem \ref{thm=OSPI}]
  We begin by proving a measure--capacity inequality. This follows the lines of the proof of theorem \ref{thm=SPIimpliesMC}. 
  In that proof, we see that the SPI is only used once, in equation \eqref{eq=utilisationLemmeSympa}, through lemma \ref{lmm=lemmeSympa}. If we apply the ``Orlicz version'' of the lemma (lemma \ref{lmm=lemmeSympaOrlicz}), we simply replace $\psi(\kappa)$ by $\theta^2(\psi(\kappa))$ and get: 
  \[
  C_\kappa \geq
   \min\left(
    \frac{C\psi(\kappa)}{\kappa} \bstar^2 ,
    \sup_{r>r_0} \frac{ 1 - \beta(r) \theta^2(\psi(\kappa))}{r}(1 - \bstar)^2 
  \right)
  \]
  Since $\lim_{\kappa \to 0} \theta^2(\psi(\kappa)) = 0$, we can conclude just as in the proof of theorem \ref{thm=SPIimpliesMC}: the measure--capacity $MC(\kappa,C_\kappa)$ holds for some $C_\kappa$ such that $\lim C_\kappa = \frac{1}{r_0}$.

  Now we can apply theorem \ref{thm=CMSPI} to get back the classical super Poincaré inequality. Note however that we only obtain it  for $r>8 r_0$.  
\end{proof}

\section{From the spectrum to super Poincaré inequalities}
\label{sec=WangRevisited}
We now  use the Orlicz-SPI developed in the previous section to give an alternative proof of the following result: 
\begin{thrm}
  \label{thm=fromSpecToSPI}
  [\cite{Wan00Essential,GW02}] Suppose that the essential spectrum of $(-L)$ is included in $[r_0, \infty)$ for some (possibly infinite) $r_0>0$. Then the super Poincaré inequality \eqref{eq=defSPI} holds, for some function $r\mapsto\beta(r)$, defined for $r>8r_0$. 
\end{thrm}
To prove the super Poincaré inequality, we shall use directly the spectral decomposition of the operator $L$, as in \cite{GW02} (section 4, see in particular theorem 4.3). 

Let $r$ and $r_1$ satisfy $r>r_1>r_0$. Since the spectrum is discrete below $\frac{1}{r_0}$, there exists $n(r)$ such that $\sigma(-L) \cap [0,1/r] = \{ 0 = \lambda_1, \lambda_2, \ldots, \lambda_{n(r)}\}$ (if $r_0 = 0$, the sequence $\lambda_k$ goes to infinity). If we denote by $P$ and $Q$ the spectral projections on $[0,r]$ and $(r,\infty)$, we get:
\begin{equation}
  \label{eq=almostOSPI}
  \int f^2 \dmu = \int (Pf)^2 \dmu + \int (Qf)^2 \dmu.
\end{equation}
The term in $Qf$ is easily bounded thanks to the spectral decomposition. If $E_\lambda$ is the resolution of identity of $-L$, 
\begin{align*}
  \int (Qf)^2 \dmu &= \int_{\lambda > 1/r} d(E_\lambda(f),f)_\mu \\
  &\leq r \int_{\lambda >1/r} \lambda d(E_\lambda(f), f)_\mu \\
  &\leq r  (f,-Lf)_\mu = r \int \abs{\nabla f}^2 \dmu.
\end{align*}

The other term is the projection of $f$ on the finite dimensional eigenspace associated to the small eigenvalues. Let $f_1, f_2 \ldots f_{n(r_1)}$ be a sequence of normalized eigenvectors. Thanks to de la Vallée Poussin's lemma, these functions are all in a space smaller than $L^2$: there is a $\PhiStar$, satisfying the hypotheses of theorem \ref{thm=OSPI}, such that $f_i \in L_{\PhiStar}$, for $i = 1,2,\ldots n(r_1)$. 

Therefore, 
\begin{align*}
  \norm{2}{Pf}^2 
  &= \sum_{i=1}^{n(r)} (f,f_i)^2 \\
  &\leq \sum_{i=1}^{n(r)} \norm{\Phi}{f}^2 \norm{\PhiStar}{f_i}^2 \\
  &\leq \left( \sum_{i=1}^{n(r)} \norm{\PhiStar}{f_i}^2 \right) \norm{\Phi}{f}^2.
\end{align*}

Going back to \eqref{eq=almostOSPI}, this yields the Orlicz-SPI: 
\begin{equation}
  \label{eq=OSPIbySpectralMeans}
  \forall r_1>r_0,\exists \Phi,  
\forall r>r_1 ,\exists \beta(r) , \quad
\int f^2 \dmu \leq r D(f,f) + \beta(r) \norm{\Phi}{f}^2,
\end{equation}
with $\beta(r) = \sum_{i = 1}^{n(r)} \norm{\PhiStar}{f_i}^2$. We have shown in the previous section how to go from such an Orlicz--SPI to a classical one: we obtain a SPI valid for all $r>8r_1$. Since $r_1$ is arbitrarily close to $r_0$, this concludes the proof of theorem \ref{thm=fromSpecToSPI}.

\section{Example}
\label{sec=example}
\subsection{Definitions}
We give here a proof of the (well-known!) logarithmic Sobolev inequality for the Ornstein Uhlenbeck semigroup (\ie{} for the Gaussian measure), using knowledge on the spectrum. 
\begin{thrm}
  Let $\gamma_d$ be the standard $d$-dimensional Gaussian measure. Then there is a $C$, independent of $d$, such that for all $f$, 
  \[
  \int f^2 \log \left(f^2 / \int f^2 \,d\gamma_d\right) \,d\gamma_d
  \leq C \int \abs{\nabla f}^2 \, d\gamma_d.
  \]
\end{thrm}
As we have written before, the use of capacity--measure inequalities prevents us from getting the optimal constant. Interestingly enough, we are however able to prove directly a dimension free inequality: $C$ does not depend on $d$.

The spectral decomposition of the Ornstein--Uhlenbeck operator is fully known, and we recall it here (details and proofs may be found e.g. in the first chapter of \cite{Bog98}). The eigenvectors are defined in terms of Hermite polynomials. In dimension $1$, these are the natural orthogonal family $(H_k)_{k\in\xN}$  associated with the scalar product $(f,g) = \int fg \dmu$; we normalize them so that $\norm{2}{H_k} = 1$.
In dimension $d$, we can define $H_\alpha = H_{\alpha_1}(x_1) \cdots H_{\alpha_d}(x_d)$ for any multiindex $\alpha\in\xN^d$. With these notations, we have the following description of the spectrum.
\begin{thrm}
  The spectrum of $-L$ is the set of integers $\xN$. In dimension $1$, the eigenvalues are simple, and $H_k$ is an eigenvector with eigenvalue $k$. In dimension $d$, the family $\{ H_\alpha, \alpha\in\xN^d, \abs{\alpha} = k\}$ form an orthonormal base of the eigenspace $E_{k}$ associated with $k$.  These eigenspaces satisfy: 
  \begin{equation}
    \dim(E_{k}) = \binom{k+d-1}{d}, \qquad \dim\left(E_{[0,k]}\right) = \dim (\bigoplus_{i=0}^k E_i) = \binom{k+d}{d} \leq 2^d k^d. 
    \label{eq=dimEigenspaces}
  \end{equation}
\end{thrm}
 To apply the results of the previous sections and get the log Sobolev inequality, we will need bounds on the Hermite polynomials.

\subsection{A bound on Hermite polynomials}
There are known bounds on the $L^p$ norms of Hermite polynomials, for fixed $p$ and $n$ large: see e.g. \cite{Lar02}. What we need here is a much cruder, but more robust bound, for all $p$ and $n$. 
\begin{thrm}
  \label{thm=LpHermite}
  There is a constant $C$ such that, for all $n$ and all $p>2$, 
  \begin{equation}
    \label{eq=LpHermiteUnidim}
    \norm{p}{H_n} \leq C^n p^{3n/4}.
  \end{equation}
  A similar estimate holds in the multivariate case:
\begin{align}
  \exists C, \forall p, \forall n, \forall d, \forall \alpha, \abs{\alpha} = n, 
  \norm{p}{H_\alpha} 
  &\leq \prod_{k=1}^d \norm{p}{H_{\alpha_k}} \notag \\
  &\leq  \prod_{k=1}^d C^{\alpha_k} p^{3\alpha_k /4} \leq C^n p^{3n/4}.
  \label{eq=LpHermite}
\end{align}
\end{thrm}
To prove this result, we need information on the asymptotics of the Hermite polynomials. Let us quote a result of Plancherel and Rotach (\cite{PR29}, see also \cite{Sze39}, \S 8.22, and \cite{Lar02}; the $\sqrt{n!}$ that appears in the latter comes from a choice of normalisation). There are three different regimes (for $x$ in the oscillating zone, for $x$ large and on the ``frontier''). On each part the asymptotic is written with an appropriate change of variables.
\begin{thrm}[ Plancherel and Rotach]
  Let $N = \sqrt{4n+2}$. 
  \begin{itemize}
    \item For $x = N\sin(\phi)$, $\abs{\phi} < \pi/2$. 
      Then 
      \begin{equation}
	\label{eq=oscPart}
	e^{-x^2/4} H_n(x) = \frac{a_n}{\cos(\phi)}
      \times \left( \sin\left( \frac{N^2}{8}(2\phi + \sin(2\phi)) - \frac{(n-1)\pi}{2}\right) + O \left(\frac{1}{n\cos^3(\phi)}\right)\right),
    \end{equation}
    where $a_n = (2/\pi)^{1/4} n^{-1/4}$. 
  \item For $x= Ncosh(\phi), 0<\phi<\infty$, 
    \begin{equation}
      e^{-x^2/4} H_n(x) = \frac{b_n}{\sqrt{\sinh(\phi)}}
      \times \exp\left( \frac{N^2}{8} (2\phi - \sinh(2\phi)) \right)
      \left( 1 + O\left( \frac{1}{n(\sinh(\phi) e^{-\phi})^3} \right)\right),
      \label{eq=nonOscPart}
    \end{equation}
    with $b_n = (8\pi)^{-1/4} n^{-1/4}$.
  \item Finally, for $x = N - 3^{-1/3}n^{-1/6} t$, with $t$ bounded, 
    \begin{equation}
      e^{-x^2/4} H_n(x) = d_n\left( A(t) + O (n^{-2/3}) \right), 
      \label{eq=frontierPart}
    \end{equation}
    where $d_n = 3^{1/3}(2/\pi^3)^{1/4} n^{-1/12}$ and $A(t)$ is the Airy function. 
  \end{itemize}
\end{thrm}
\begin{proof}
  In this proof, $C$ ($c$) is a large (small) constant, independent of $n,p,x$, that may change from line to line. 
  We want to bound $\int H_n(x)^p d\gamma(x)$. We decompose this integral in three parts: 
  \begin{align*}
  I_1 &= \int_{\abs{x} \leq N - n^{-1/6}} H_n(x)^p d\gamma(x), \\
  I_2  &= \int_{\abs{x} \in [N - n^{-1/6}, N+ n^{-1/6}]} H_n(x)^pd\gamma(x), \\
  I_3  &= \int_{\abs{x} \geq N + n^{-1/6}} H_n(x)^pd\gamma(x).
  \end{align*}
  Let us begin by the oscillating part $I_1$. It is easy to see that, since $x\leq N - n^{-1/6}$, $\cos(\phi)\geq cn^{-1/3}$. Therefore $\frac{1}{\cos^3(\phi) n}$ is bounded (by some $C$ independent of $x,n$), and
  \begin{align*}
  \exists C, \forall n, \forall x\leq N - n^{-1/6}, \quad
  e^{-x^2/4} \abs{H_n(x)} &\leq  C \frac{a_n}{\cos(\phi)} \\
  &\leq C n^{-1/4}n^{1/3}  \\
  &\leq C n^{1/12}.
  \end{align*}
  We get
  \begin{align}
    I_1 &\leq \int_{[-N,N]} C^p e^{px^2/4} n^{p/12} e^{-x^2/2} dx \notag \\
    &\leq C^p n^{1/2 + p/12} e^{p(n+ 1/2)}.
    \label{eq=onIOne}
  \end{align}
  Similarly, the ``frontier'' part $I_2$ is easily bounded (thanks to \eqref{eq=frontierPart}) by
  \begin{equation}
    \label{eq=onITwo}
    I_2 \leq C n^{-p/12}e^{p(n+1)}.
  \end{equation}
  The main contribution comes from the non oscillating part $I_3$. In $I_3$, $x \geq N + n^{-1/6}$, so $\cosh(\phi) \geq 1+ n^{-2/3}$. Therefore, $\sinh(\phi) \geq n^{-1/3}$. For large $\phi$, $\sinh(\phi) e^{-\phi}$ is bounded away from zero. All this shows that we can simplify the bound \eqref{eq=nonOscPart} and get: 
  \begin{align*}
  e^{-x^2/4} H_n(x) \leq C n^{-1/4} n^{1/6} 
         \exp \left( \frac{N^2}{8} (2\phi - \sinh(2\phi)) \right) \\
  \leq C \exp \left( \frac{N^2}{8} (2\phi - \sinh(2\phi)) \right)
\end{align*}
Since $x = N\cosh \phi$, we can put the $e^{-x^2/4}$ on the r.h.s. and get: 
  \begin{align*}
  H_n(x) &\leq C 
         \exp \left( \frac{N^2}{8} (2\phi + 2 \cosh(\phi)^2 - \sinh(2\phi)) \right) \\
	 &\leq C
	 \exp \left( \frac{N^2}{8} (2\phi + 1 + \exp(-2\phi))\right)\\
	 &\leq C\exp (n + 1/2) \exp \left( (n+1/2) \phi\right)
\end{align*}
Finally $\phi \leq \log( 2 \cosh (\phi)) \leq \log(2x/N)$, so
\begin{equation*}
  H_n(x) \leq C\left(\frac{2 e x}{N}\right)^{n + 1/2}
\end{equation*}
for $x\geq N + n^{-1/6}$. Going back to $I_3$, we get 
\begin{align*}
  I_3 &\leq C^p \int_x  \left( \frac{2e x}{N} \right)^{(n + 1/2)p} e^{-x^2/2} dx \\
  &\leq C^{np} \frac{1}{N^{(n+1/2)p}} \int_x x^{(n+1/2)p} d\gamma(x).
\end{align*}
The moments of the standard Gaussian measure satisfy $\int x^{2k} d\gamma(x)
\leq 2^k k! \leq C k^{k}$. This yields (recalling that $N = \sqrt{4n+2}$)
\begin{align}
I_3 
  &\leq C^{np} \frac{1}{N^{(n+1)p}}  \left( \frac{ (n+1/2) p}{2}\right)^{(n+1/2)p/2}  \notag\\
  &\leq  C^{np} p^{(n+1/2)p/2}.
  \label{eq=onIThree}
\end{align}
The three bounds \eqref{eq=onIOne}, \eqref{eq=onITwo} and \eqref{eq=onIThree} on $I_1, I_2$ and $I_3$ imply: 
\[
  \norm{p}{H_n} \leq (I_1 + I_2 + I_3)^{1/p}
  \leq C^n p^{(n+1/2)/2} \leq C^n p^{3n/4},
\]
therefore \eqref{eq=LpHermiteUnidim} holds. The multidimensional estimate \eqref{eq=LpHermite} is a direct consequence of this. This concludes the proof of theorem \ref{thm=LpHermite}.
\end{proof}

\subsection{The log Sobolev inequality}
The spectral decomposition easily  gives us a Orlicz-SPI, thanks to the reasoning of section~\ref{sec=WangRevisited}. Indeed, the $H_\alpha$ are in $L^p(\mu)$, for any $p<\infty$ (they are polynomials and  $\mu$ is the standard gaussian). For any $p>2$, let $q$ be the conjugate exponent ($p^{-1} + q^{-1} = 2$). We get \eqref{eq=OSPIbySpectralMeans}: 
\begin{equation}
  \label{eq=OSPIq}
\int f^2 \dmu \leq r \int \abs{\nabla f}^2 \dmu + \beta(r) \norm{q}{f}^2,
\end{equation}
where $\beta(r) = \sum_{\alpha, \abs{\alpha} \leq 1/r} \norm{p}{H_\alpha}^2$. 
Using the bounds \eqref{eq=dimEigenspaces} and \eqref{eq=LpHermite} from previous section, we get:
To sum up, the factor $\beta(r)$  that appears in \eqref{eq=OSPIq} is bounded by: 
\begin{align*}
  \beta(r) &\leq \sum_{\alpha, \abs{\alpha} \leq m(r)} \norm{p}{H_\alpha}^2 \\
  &\leq \dim (E_{[0,1/r]}) C^{2/r} p^{3/(2r)} \\
  &\leq 2^d\frac{1}{r^d} C^{2/r} p^{3/(2r)}.
\end{align*}

To prove the log Sobolev inequality, we use once more the measure-capacity criterion. Let us apply theorem \ref{thm=OSPI} to our case. We get a measure capacity inequality, with
\[
C_\kappa = 
   \min\left(
    \frac{C_P\psi(\kappa)}{\kappa} \bstar^2 ,
    \sup_{r>0} \frac{ 1 - \beta(r) \theta^2(\psi(\kappa))}{r}(1 - \bstar)^2 
  \right),
  \]
  where $C_P$ is the a priori bound on the constant $C_\kappa$, given by the Poincaré inequality (here the spectrum is known, so we can choose $C_P = 1$). Since we are not interested in optimal constants here, we fix $b^\star = 1/2$. 
  In the proof of theorem \ref{thm=OSPI} we  choose $\psi(x) = x\log(1/x)$: the estimate becomes
  \begin{equation}
    \label{eq=minCKappa}
  C_\kappa = \frac{1}{4}
   \min\left(
    \log(1/\kappa) ,
    \sup_{r>0} \frac{ 1 - \beta(r) \theta^2(\psi(\kappa))}{r}
  \right)
\end{equation}
  We now make the following 
  \begin{claim}
    There exists $\kappa_1$, and, for all $\kappa <\kappa_1$, a $r_\kappa$, such that
    \[
    \forall \kappa < \kappa_1, \forall d, \exists p, \quad
    \beta(r_\kappa) \theta^2 (\psi(\kappa)) \leq 1/2.
    \]
    One can choose $r_\kappa = (\frac{1}{4} \log(1/\kappa))^{-1}$.
  \end{claim}
  \begin{proof}
    [Proof of the claim]
  We have chosen to work with $L^p$ norms, this corresponds to $\PhiStar(x) = \frac{x^p}{p}$, $\Phi(x) = \frac{x^q}{q}$. Therefore $\theta^2(x) = 4p x^{p/2}$, and we get:
  \begin{align*}
    \beta(r) \theta^2(\psi(\kappa)) 
    &\leq \left(\frac{2}{r}\right)^d C^{2/r} p^{3/(2r)} \cdot
      4p \frac{\kappa^{p/2}}{\log(1/\kappa)^{p/2}}.
  \end{align*}
  Now we choose a specific $r=r_\kappa$, which does not depend on $d$ and $p$: $r_\kappa = (\log(1/\kappa)/4)^{-1}$. We impose $p>C^2$.  With this choice,   $C^{2/r}p^{3/(2r)}\leq p^{5r/2} = \exp\left( (5/2)\log(p) \log(1/\kappa) \right) \leq \exp\left( \frac{5p}{12} \log(1/\kappa)\right) =  \kappa^{-5p/12}$ (we used the elementary inequality $\log(p) \leq p/6$). The inequality becomes:
  \begin{align*}
    \beta(r) \theta^2(\psi(\kappa)) 
    &\leq \left(\frac{2}{r}\right)^d  \kappa^{-5p/12}  \cdot
      4p \frac{\kappa^{p/2}}{\log(1/\kappa)^{p/2}} \\
    &\leq 2^{-d} \log(1/\kappa)^{d - p/2} \cdot 4p \kappa^{p/12}.
  \end{align*}
  Now, $p\kappa^{p/12} \leq C \kappa^{p/13}$, where $C$ does not depend on $p$, so
  \[
    \beta(r) \theta^2(\psi(\kappa)) 
    \leq C \log(1/\kappa)^{d- p/2} \kappa^{p/13}.
    \]
    Finally, there is a $\kappa_0$ (independent of $p,d$) such that, for $\kappa<\kappa_0$, $C\leq\log(1/\kappa)$:
  \[
    \beta(r) \theta^2(\psi(\kappa)) 
    \leq \log(1/\kappa)^{d + 1 - p/2} \kappa^{p/13} 
    \leq \log(1/\kappa)^{d + 1 - p/2} \kappa^{2/13}.
    \]
    For $p> 2d + 2$, and $\kappa < \kappa_1$ (which does not depend on $(p,d)$), 
    \[ \beta(r) \theta^2 (\psi(\kappa)) \leq \frac{1}{2},\]
    and the claim is proved.
  \end{proof}
  Going back to \eqref{eq=minCKappa}, we can bound the sup in the second term by the value of the argument for $r = r_\kappa$: 
  \begin{align*}
  C_\kappa &\geq \frac{1}{4} \min \left( \log(1/\kappa), \frac{1}{2 r_\kappa}\right) \\
  &\geq \frac{1}{32} \log(1/\kappa).
\end{align*}
Such an inequality is known to imply the log Sobolev inequality (see \eg{} \cite{BR03}, part 4.6). The dimension does not appear in the last inequality, so we obtain the same (non optimal) constant for any dimension $d$.

\bibliographystyle{amsalpha}
\bibliography{biblio}

\providecommand{\bysame}{\leavevmode\hbox to3em{\hrulefill}\thinspace}
\providecommand{\MR}{\relax\ifhmode\unskip\space\fi MR }
\providecommand{\MRhref}[2]{%
  \href{http://www.ams.org/mathscinet-getitem?mr=#1}{#2}
}
\providecommand{\href}[2]{#2}
\begin{thebibliography}{Wan00b}

\bibitem[ABD07]{ABD07}
Anton Arnold, Jean-Philippe Bartier, and Jean Dolbeault, \emph{Interpolation
  between logarithmic {S}obolev and {P}oincar\'e inequalities}, Commun. Math.
  Sci. \textbf{5} (2007), no.~4, 971--979. \MR{MR2375056 (2009c:60213)}

\bibitem[BCR06]{BCR05b}
Franck Barthe, Patrick Cattiaux, and Cyril Roberto, \emph{Interpolated
  inequalities between exponential and gaussian, {O}rlicz hypercontractivity
  and application to isoperimetry}, Rev. Mat. Iberoamericana \textbf{22}
  (2006), no.~3, 993--1067.

\bibitem[BCR07]{BCR07}
\bysame, \emph{Isoperimetry between exponential and {G}aussian}, Electron. J.
  Probab. \textbf{12} (2007), no. 44, 1212--1237 (electronic). \MR{MR2346509
  (2008j:60046)}

\bibitem[Bog98]{Bog98}
Vladimir~I. Bogachev, \emph{Gaussian measures}, Mathematical Surveys and
  Monographs, vol.~62, American Mathematical Society, Providence, RI, 1998.
  \MR{MR1642391 (2000a:60004)}

\bibitem[BR03]{BR03}
Franck Barthe and Cyril Roberto, \emph{Sobolev inequalities for probability
  measures on the real line}, Studia Math. \textbf{159} (2003), no.~3,
  481--497, Dedicated to Professor Aleksander Pe\l czy\'nski on the occasion of
  his 70th birthday (Polish). \MR{MR2052235 (2006c:60019)}

\bibitem[GW02]{GW02}
Fu-Zhou Gong and Feng-Yu Wang, \emph{Functional inequalities for uniformly
  integrable semigroups and application to essential spectrums}, Forum Math.
  \textbf{14} (2002), no.~2, 293--313. \MR{MR1880915 (2003a:47097)}

\bibitem[LC02]{Lar02}
Lars Larsson-Cohn, \emph{{$L\sp p$}-norms of {H}ermite polynomials and an
  extremal problem on {W}iener chaos}, Ark. Mat. \textbf{40} (2002), no.~1,
  133--144. \MR{MR1948890 (2004a:60056)}

\bibitem[Maz85]{Maz85}
Vladimir~G. Maz'ja, \emph{Sobolev spaces}, Springer Series in Soviet
  Mathematics, Springer-Verlag, Berlin, 1985, Translated from the Russian by T.
  O. Shaposhnikova. \MR{MR817985 (87g:46056)}

\bibitem[PR29]{PR29}
M.~Plancherel and W.~Rotach, \emph{Sur les valeurs asymptotiques des polynomes
  d'{H}ermite {$H\sb n(x)=(-I)\sp n e\sp {\frac{{x\sp 2}}{2}} \frac{{d\sp n
  }}{{dx\sp n }}\left({e\sp {-\frac{{x\sp 2}}{2}}}\right)$}}, Comment. Math.
  Helv. \textbf{1} (1929), no.~1, 227--254. \MR{MR1509395}

\bibitem[RW01]{RW01}
Michael R{\"o}ckner and Feng-Yu Wang, \emph{Weak {P}oincar{\'e} inequalities
  and {$L^2$} convergence rates of {M}arkov semigroups}, Journal of Functional
  Analysis \textbf{185} (2001), 564--603.

\bibitem[RZ07]{RZ07}
Cyril Roberto and Bogus{\l}aw Zegarli{\'n}ski, \emph{Orlicz-{S}obolev
  inequalities for sub-{G}aussian measures and ergodicity of {M}arkov
  semi-groups}, J. Funct. Anal. \textbf{243} (2007), no.~1, 28--66.
  \MR{MR2289793 (2008g:52015)}

\bibitem[Sze39]{Sze39}
Gabor Szeg{\"o}, \emph{Orthogonal {P}olynomials}, American Mathematical
  Society, New York, 1939, American Mathematical Society Colloquium
  Publications, v. 23. \MR{MR0000077 (1,14b)}

\bibitem[Wan00a]{Wan00Essential}
Feng-Yu Wang, \emph{Functional inequalities for empty essential spectrum}, J.
  Funct. Anal. \textbf{170} (2000), no.~1, 219--245. \MR{MR1736202
  (2001a:58043)}

\bibitem[Wan00b]{Wan00Semigroup}
\bysame, \emph{Functional inequalities, semigroup properties and spectrum
  estimates}, Infin. Dimens. Anal. Quantum Probab. Relat. Top. \textbf{3}
  (2000), no.~2, 263--295. \MR{MR1812701 (2002b:47083)}

\bibitem[Zit08]{Zit07annealing}
Pierre-Andr{\'e} Zitt, \emph{Annealing diffusions in a potential with a slow
  growth}, Stochastic Processes and their Applications \textbf{118} (2008),
  no.~1, 76--119, doi: 10.1016/j.spa.2007.04.002.

\end{thebibliography}
\end{document}